\documentclass[12pt,twoside,reqno]{amsart}

\newtheorem{thm}{Theorem}[section]
\newtheorem{lemma}{Lemma}[section]

\def \R{{\Bbb R}}

\numberwithin{equation}{section}

\begin{document}

\title[Exponential decay for the ZK equation]
{Exponential decay for the linear Zakharov-Kuznetsov equation
without critical domains restrictions}
\author[
G.~G. Doronin,
\ N.~A. Larkin]
{
G.~G. Doronin,
\ N.~A. Larkin$^\S$
\bigskip
\\
{\tiny
Departamento de Matem\'atica,\\
Universidade Estadual de Maring\'a,\\
87020-900, Maring\'a - PR, Brazil.
}
}
\address
{
Departamento de Matem\'atica\\
Universidade Estadual de Maring\'a\\
87020-900, Maring\'a - PR, Brazil.
}
\email{ggdoronin@uem.br \ \  nlarkine@uem.br}
\date{}


\thanks{$^\S$Corresponding author; partially supported by Funda\c{c}\~ao Arauc\'aria}

\begin{abstract}
Initial-boundary value problems for the linear Zakharov-Kuznetsov
equation posed on bounded rectangles are considered. Spectral
properties of a stationary operator are studied in order to show
that the evolution problem posed on a bounded rectangle has no
critical restrictions on its size. Exponential decay of regular
solutions is established.
\end{abstract}

\maketitle

\section{Introduction}\label{introduction}

We are concerned with initial-boundary value problems (IBVPs) posed
on bounded rectangles for the Zakharov-Kuznetsov (ZK) equation
\begin{equation}\label{zk}
u_t+u_x+uu_x +u_{xxx}+u_{xyy}=0
\end{equation}
which is a two-dimensional
analog of the well-known Korteweg-de Vries (KdV) equation
\begin{equation}\label{kdv}
u_t+uu_x+u_{xxx}=0
\end{equation}
with clear plasma physics applications \cite{zk}.

Equations \eqref{zk} and \eqref{kdv} are typical examples
of so-called
dispersive equations which jointly with the Schr\"{o}dinger and
other equations attract considerable attention of both pure and
applied mathematicians in the past decades. The KdV equation is
probably more studied in this context. The theory of the
initial-value problem for \eqref{kdv} posed on the whole real line
is considerably advanced today \cite{bona2,kato,ponce2} and the
references therein.

Recently, due to physics and numerics needs, publications on
initial-boundary value problems for dispersive equations have been
appeared \cite{bona1,larkin,lar2}. In particular, it has been
discovered that the KdV equation posed on a bounded interval
possesses an implicit internal dissipation. This allowed to prove
the exponential decay rate of small solutions for \eqref{kdv} posed
on bounded intervals without adding any artificial damping term
\cite{zuazua}. Similar results were proved for a wide class of
dispersive equations of any odd order with one space variable
\cite{familark}.

However, \eqref{kdv} is a satisfactory approximation for real waves
phenomena while the equation is posed on the whole line
($x\in\mathbb{R}$); if cutting-off domains are taken into account,
\eqref{kdv} is no longer expected to mirror an accurate rendition of
reality. More correct equation in this case (see, for instance,
\cite{bona1}) should be written as
\begin{equation}\label{1.3}
u_t+ u_x+uu_x+u_{xxx}=0.
\end{equation}
Indeed, if $x\in\R,\ t>0$,
the linear traveling term $u_x$ in \eqref{1.3} can be easily scaled
out by a simple change of variables; but it can
not be safely ignored for problems posed on both finite and semi-infinite intervals without
changes in the original domain.

Once bounded domains are considered
as a spatial region of waves propagation, their sizes appear to be
restricted by certain critical conditions. An important result
regarding these conditions is the explicit description
of a spectrum-related countable critical set \cite{rosier}
\begin{equation}\label{critKdV}
\mathcal{N}=\left\{\frac{2\pi}{\sqrt3}\sqrt{k^2+kl+l^2}\,;\ \ \
k,l\in\mathbb{N}\right\}.
\end{equation}
While studying the controllability and stabilization of solutions
for \eqref{1.3}, the set $\mathcal{N}$ provides qualitative
difficulties when the length of a spatial interval coincides with
some of its elements \cite{rosier}.
More recent results on control and stabilizability for the KdV equation can
be found in \cite{rosier1,rozan}.

Quite recently, the interest on dispersive equations became to be
extended to the multi-dimensional models such as
Kadomtsev-Petviashvili (KP) and ZK equations. As far as the ZK
equation is concerned, the results on both IVP and IBVP can be found
in
\cite{faminski,faminski2,farah,pastor,pastor2,saut,temam2}. Our work
has been inspired by \cite{temam} where \eqref{zk} has been
considered on a strip bounded in the $x$ variable. Studying this
paper, we have found that the term $u_{xyy}$ in \eqref{zk} delivers
additional dissipation which may ensure decay of solutions. For
instance, the term $u_{xyy}$ provides exponential decay of small
solutions in a channel-type domain; namely, in a half-strip
unbounded in $x$ direction \cite{larkintronco}. However, there are
restrictions on a width of a channel.

It has been showed in \cite{dorlar} that the above restrictions are stipulated by the spectral
properties of the corresponding stationary operator. More precisely, considering linearized \eqref{zk} posed on a rectangle
$$
\mathcal{D}=(0,L)\times(0,B)\subset \mathbb{R}^2
$$
with the simplest Dirichlet-type boundary data, one can see that stabilizability of solutions fails if $L>0$ and $B>0$ solve
\begin{equation}\label{critZK}
\left(\frac{2\pi}{L\sqrt3}\sqrt{k^2+kl+l^2}\right)^2+\left(\frac{\pi n}{B}\right)^2=1,
\end{equation}
i.e., if $\mathcal{D}$ is of a {\it critical} size, likewise in the
case of KdV posed on an interval. In other words, \eqref{critZK} is
a 2D generalization of \eqref{critKdV}.

The following question arises naturally:
\begin{itemize}
\item
Are there some physically reasonable mechanisms which help to avoid
the critical restrictions for the ZK equation?
\end{itemize}

In the present paper we show that there are specific physically
reasonable boundary conditions such that corresponding IBVP with no
size restrictions on a domain possess solutions that decay
exponentially, at least in a linear framework. We exploit
effectively the dissipative role of the term $u_{xyy}$ which
apparently is due to the elliptic properties of a stationary
operator considering as applied to $u_x.$

The main goal of our paper is to establish the existence and
uniqueness of global-in-time regular solutions of linearized
\eqref{zk} posed on bounded rectangles with a special type boundary
condition on the part $\{y=B\}$ of a spatial domain,
and the exponential decay rate of these solutions independent of
critical size limitations.

The paper has the following structure. Section 1 is Introduction.
Section \ref{problem} contains formulation of the problem and
auxiliaries. In Section \ref{existence}, we prove the existence
theorem and preliminary estimates. In Section \ref{decay}, we
provide forthcoming estimates to establish our principal
stabilization result.

\section{Problem and preliminaries}\label{problem}

Let $L,B,T$ be finite positive numbers. Define
\begin{equation*}
\mathcal{D}=\{(x,y)\in\mathbb{R}^2: \ x\in(0,L),\ y\in(0,B) \},\ \ \
\mathcal{D}_T=\mathcal{D}\times (0,T).
\end{equation*}

We consider in $\mathcal{D}_T$ the following IBVP:
\begin{align}
L^{\mathcal{D}}&u\equiv u_t+u_x+u_{xxx}+u_{xyy}=0,\ \ \text{in}\
\mathcal{D}_T; \label{2.1}
\\
&u(x,0,t)=0,\ u(x,B,t)=u_{xy}(x,B,t),\ \ x\in(0,L),\ t>0;
\label{2.2}
\\
&u(0,y,t)=u(L,y,t)=u_x(L,y,t)=0,\ \ y\in(0,B),\ t>0;
\label{2.3}
\\
&u(x,y,0)=u_0(x,y),\ \ (x,y)\in\mathcal{D},
\label{2.4}
\end{align}
where $u_0:\mathcal{D}\to\mathbb{R}$ is a given function.

Hereafter subscripts $u_x,\ u_{xy},$ etc. denote the partial
derivatives, as well as $\partial_x$ or $\partial_{xy}^2$ when it is
convenient. Operators $\nabla$ and $\Delta$ are the gradient and
Laplacian acting over $\mathcal{D}.$ By $(\cdot,\cdot)$ and
$\|\cdot\|$ we denote the inner product and the norm in
$L^2(\mathcal{D}),$ and $\|\cdot\|_{H^k}$ stands for the norm in the
$L^2$-based Sobolev spaces.

The following result will be useful:
\begin{lemma}\label{lemma1}
For arbitrary $L>0,\ B>0$ let $u:[0,L]\times[0,B]\to\mathbb{C}$ be a
regular solution to the eigenvalue problem
\begin{align}
&u_x+u_{xxx}+u_{xyy}=\lambda u,\ \ (x,y)\in \mathcal{D},\ \
\lambda\in\mathbb{C};\label{2.9}
\\
&u(x,0)=0,\ u(x,B)=u_{xy}(x,B),\ \ x\in (0,L); \label{2.10}
\\
&u(0,y)=u_x(0,y)=u(L,y)=u_x(L,y)=0,\ \ y\in(0,B). \label{2.11}
\end{align}
Then $u\equiv 0.$
\end{lemma}

\begin{proof}
Performing $y\mapsto B-y$ and continuing $u$ by zero to all
$x\in\mathbb{R},$ the result for $\lambda =0$ follows by Holmgren's
uniqueness theorem \cite{bers}. If $\lambda \ne 0,$ the function
$$
\widehat{u}(\xi,y)=\frac{1}{\sqrt{2\pi}}
\int_{\mathbb{R}}\overline{u}(x,y)e^{-i\xi x}\,dx,
$$
where
$$
\overline{u}:\mathbb{R}\times\mathbb{R}^+ \to\mathbb{R},\
\overline{u}(x,y)=1_{[0,L]}u(x,y)
$$
solves
$$
\widehat{u}_{yy}+\left(1-\xi^2-\frac{\lambda}{i\xi}\right)\widehat{u}=0,\
\widehat{u}|_{y=0}=\widehat{u}_y|_{y=0}=0.
$$
Therefore, $\widehat{u}\equiv 0$ which gives $u\equiv 0,$ as well.
\end{proof}

\section{Existence theorem}\label{existence}

In this section we state the existence results for problems in
$\mathcal{D}_T.$
\begin{thm}\label{theorem1}
Let $u_0\in H^3(\mathcal{D})$ be a given function such that
$$
u_0|_{y=B}=u_{0xy}|_{y=B},\
u_0|_{y=0}=u_0|_{x=0,L}=u_{0x}|_{x=L}=0.
$$
Then for all finite
positive $L,B,T$ there exists a unique regular solution to
\eqref{2.1}-\eqref{2.4} such that
\begin{align*}
&u\in L^{\infty}(0,T;H^3(\mathcal{D}));\\
&u_t\in L^{\infty}(0,T;L^2(\mathcal{D}))
\end{align*}
and for all $t\in(0,T)$ it holds
\begin{align}\label{33.20}
\|u\|^2(t)&+\int_0^t\Bigl[\|u_x\|^2(s)+\|u_y\|^2(s)\Bigr]\,ds\notag\\
&+\int_0^t\left\{\int_0^Bu_x^2(0,y,s)\,dy+2\int_0^L(1+x)u^2(x,B,s)\,dx\right\}\,ds\notag\\
&\le (1+L+T)\|u_0\|^2.
\end{align}
\end{thm}

\begin{proof}
To prove this theorem we use the classical semigroup approach and appropriate boundary estimates.
First we write \eqref{2.1}-\eqref{2.4} as an abstract evolution equation
\begin{equation}\label{AbstrEq}
\frac{d}{dt}u=Au
\end{equation}
subject to the initial condition
\begin{equation}\label{AbstrId}
u(0)=u_0.
\end{equation}
Consider the space
\begin{align*}
D(A)=\Bigl\{ &v\in H^{3}(\mathcal{D}):\ v(x,0)=0,\
v(x,B)=v_{xy}(x,B),\\
&v(0,y)=v_x(0,y)=v(L,y)=v_x(L,y)=0\Bigr\}
\end{align*}
and the closed linear operator $A:D(A)\rightarrow
L^{2}(\mathcal{D})$ defined by
$$
Av=-v_x-\Delta v_x.
$$
Let $v\in D(A)  .$ Then%
\[
\left(Av,v\right)  =
-\frac12\int_0^Bv_x^2(0,y)\,dy-\int_0^Lv^2(x,B)\,dx \leq0.
\]
On the other hand, for the adjoint operator $A^{\ast}$ defined as
$$
A^{\ast}w=\Delta w_x+w_x
$$
with the domain
\begin{align*}
D(A^{\ast})=\Bigl\{ &w\in H^{3}(\mathcal{D}):\ w(x,0)=0,\
w(x,B)=-w_{xy}(x,B),\\
&w(0,y)=w_x(0,y)=w(L,y)=w_x(L,y)=0\Bigr\}
\end{align*}
it holds
$$
\left(w,A^{\ast}w\right)
=-\frac{1}{2}\int_0^Bw_x^2(L,y)\,dy-\int_0^Lw^2(x,B)\,dx \leq 0
$$
which means that both $A$ and $A^{\ast}$ are dissipative. By
semigroup theory (see, for instance, \cite{pazy}), the operator $A$
generates a strongly continuous semigroup of contractions $\left\{
S\left(  t\right)  \right\} _{t\geq0}$ on $L^{2}\left(
\mathcal{D}\right)  .$ Then for all $u_{0}\in D(A)$ there exists a
unique solution $u(t) =S\left( t\right)u_{0}$ for \eqref{AbstrEq},
\eqref{AbstrId} satisfying
$$
u\in C\bigl( [0,T]  ;D(A) \bigr)\cap
C^1\bigl((0,T);L^2(\mathcal{D})\bigr)
$$
and
\begin{equation*}
\left\Vert u\right\Vert _{C([0,T];D(A))}\leq\left\Vert u_{0}\right\Vert_{D(A)} .\label{f53}%
\end{equation*}
Due to the structure of \eqref{2.1}, one concludes that the second
inclusion above is in fact closed, i.e. $u\in
C^1\left([0,T];L^2(\mathcal{D})\right).$ To see that $u(x,y,t)$
satisfies \eqref{33.20}, we need the following estimates.

\subsection{Estimate I}\label{1-st estimate}
Multiply \eqref{2.1} by $u$ and integrate over $\mathcal{D}$ to
obtain
\begin{align*}\label{3.5}
(L^{\mathcal{D}}u,u)&\equiv \frac12\frac{d}{dt}\|u\|^2(t)\notag\\
&+\frac12\int_0^Bu_x^2(0,y,t)\,dy+\int_0^Lu^2(x,B,t)\,dx=0 ,\ \
t\in(0,T).
\end{align*}
Integrating in $t\in(0,T)$ gives
\begin{equation}\label{f54}
\|u\|(t)\leq\|u_{0}\|, \ \ \ t\in[0,T]
\end{equation}
and
\begin{align}
\Phi(T,L,B)&\equiv \frac12\int_0^T\int_0^Bu_x^2(0,y,t)\,dy\,dt+\int_0^T\int_0^Lu^2(x,B,t)\,dx\,dt\notag\\
&\le \frac12\|u_0\|^2.\label{f55}%
\end{align}

\subsection{Estimate II}\label{2-nd estimate}
Multiplying \eqref{2.1} by $(1+x)u,$ we get
\begin{align}\label{3.6}
\left(L^{\mathcal{D}}u,(1+x)u\right)&\equiv
\frac12\frac{d}{dt}\int_0^L\int_0^B(1+x)u^2\,dx\,dy\notag\\
&+\int_0^L\int_0^B\left[\frac32u_x^2+\frac12u_y^2-\frac12u^2\right]\,dx\,dy\notag\\
&+\frac12\int_0^Bu_x^2(0,y,t)\,dy+\int_0^L(1+x)u^2(x,B,t)\,dx=0.
\end{align}
Integrating in $t\in(0,T),$ \eqref{3.6} becomes
\begin{align}\label{3.7}
\|\sqrt{1+x}\,&u\|^2(t)+\int_0^t\Bigl[3\|u_x\|^2(s)+\|u_y\|^2(s)\Bigr]\,ds\notag\\
&+\int_0^t\left\{\int_0^Bu_x^2(0,y,s)\,dy+2\int_0^L(1+x)u^2(x,B,s)\,dx\right\}\,ds\notag\\
&=\|\sqrt{1+x}\,u_0\|^2+\int_0^t\|u\|^2(s)\,ds
\end{align}
which yields \eqref{33.20}.

\end{proof}

\section{Exponential decay}\label{decay}

To prove the exponential decay of the $L^2$ norm of solutions, we
need the observability inequality and the uniform estimate for
solution in $L^2(0,T;L^2).$

\subsection{Estimate III}\label{observability}
We multiply \eqref{2.1} by $(T-t)u$ and integrate over
$\mathcal{D}_T$ to obtain
\begin{align*}\label{4.1}
\int_0^T\left(L^{\mathcal{D}}u,(T-t)u\right)\,dt&\equiv
\frac12\int_0^T\|u\|^2(t)\,dt-\frac{T}{2}\|u_0\|^2\notag\\
&+\frac12\int_0^T(T-t)\int_0^Bu_x^2(0,y,t)\,dy\,dt\notag\\
&+\int_0^T(T-t)\int_0^Lu^2(x,B,t)\,dx\,dt=0.
\end{align*}
Hence,
\begin{equation}\label{4.2}
\|u_0\|^2\le \frac1{T}\int_0^T\|u\|^2(t)\,dt+2\Phi(T,L,B)
\end{equation}
with $\Phi(T,L,B)$ defined in \eqref{f55}.

\subsection{Estimate IV}\label{4-th est}
In conditions of Theorem \ref{theorem1} it holds
\begin{equation}\label{4.3}
\frac1{T}\int_0^T\|u\|^2(t)\,dt\le M\Phi(T,L,B),
\end{equation}
where a constant $M$ does not depend on $T>0.$
\begin{proof}
Indeed, if \eqref{4.3} is false, then there exists a sequence
$u_n(x,y,t)$ of solutions to \eqref{2.1}-\eqref{2.4} such that
\begin{equation}\label{4.4}
\int_0^T\|u_n\|^2(t)\ge n\Phi_n(T,L,B)
\end{equation}
with
$$
\Phi_n(T,L,B)\equiv\int_0^T\left\{\frac12\int_0^Bu_{nx}^2(0,y,t)\,dy+\int_0^Lu_n^2(x,B,t)\,dx\right\}\,dt.
$$
Since $\int_0^T\|u_n\|^2\,dt\le T\|u_0\|^2,$ then
$\Phi_n(T,L,B)\longrightarrow 0$ as $n\to\infty.$ By the properties
of solutions there exists a subsequence
$\Phi_{n_k}\longrightarrow\Phi.$ Therefore, $\Phi(T,L,B)\equiv 0$
which implies
$$
u_x(0,y,t)=u(x,B,t)=0.
$$
Due to Lemma \ref{lemma1} this means $u\equiv 0$ in $\mathcal{D}_T.$
This contradicts to $u_0$ being arbitrary, and \eqref{4.3} is
thereby true.
\end{proof}

We now prove the main result of this work.
\begin{thm}\label{theorem2}
Let all the conditions of Theorem \ref{theorem1} hold. Then
$$
\|u\|(t)\le K\|u_0\|\exp\{-\gamma t\}\ \ for\ all\ \ t>0.
$$
\end{thm}
\begin{proof}
Combing \eqref{4.2} and \eqref{4.3} one has
$$
\|u_0\|^2\le C\Phi\text{  with  }C=M+2.
$$
Therefore,
$$
(1+C)\|u\|^2(t)=(1+C)\left[\|u_0\|^2-2\Phi\right]\le
C\|u_0\|^2-(2+C)\Phi\le C\|u_0\|^2.
$$
Thus $$\|u\|(t)\le K\|u_0\|\exp\{-\gamma t\}$$
with
$$
K=\frac{1+C}{C} \text{ and } \gamma =-\ln \frac{C}{1+C}.
$$
\end{proof}


\medskip



\begin{thebibliography}{99}

\bibitem{bers}
\newblock L. Bers, F. John and M. Schecter,
\newblock Partial differential equations.
\newblock John Wiley \& Sons Inc., New York-London-Sydney, 1964.

\bibitem{bona2}
\newblock J. L. Bona and R. W. Smith,
\newblock The initial-value problem for the Korteweg-de Vries equation,
\newblock Phil. Trans. Royal Soc. London Series A 278 (1975), 555--601.

\bibitem{bona1}
\newblock J. L. Bona, S. M. Sun and B.-Y. Zhang,
\newblock A nonhomogeneous boundary-value problem for the Korteweg-de Vries equation posed on a finite domain,
\newblock Comm. Partial Differential Equations 28 (2003), 1391--1436.


\bibitem{doronin1}
\newblock G. G. Doronin and N. A. Larkin,
\newblock KdV equation in domains with moving boundaries,
\newblock J. Math. Anal. Appl. 328 (2007), 503--515.

\bibitem{dorlar}
\newblock G. G. Doronin and N. A. Larkin,
\newblock Stabilization of regular solutions for the Zakharov-Kuznetsov equation posed on bounded rectangles and on a strip,
\newblock arXiv:submit/0558971 [math.AP] (2012).

\bibitem{faminski}
\newblock A. V. Faminskii,
\newblock The Cauchy problem for the Zakharov-Kuznetsov equation (Russian),
\newblock Differentsial'nye Uravneniya, 31 (1995), 1070--1081;
\newblock Engl. transl. in: Differential Equations 31 (1995), 1002--1012.

\bibitem{faminski2}
\newblock A. V. Faminskii,
\newblock Well-posed initial-boundary value problems for the Zakharov-Kuznetsov equation,
\newblock Electronic Journal of Differential equations 127 (2008), 1--23.

\bibitem{familark}
\newblock A. V. Faminskii and N. A. Larkin,
\newblock Initial-boundary value problems for quasilinear dispersive equations posed on a bounded interval,
\newblock Elec. J. Diff. Equations 2010 (2010), 1--20.

\bibitem{farah}
\newblock L. G. Farah, F. Linares and A. Pastor,
\newblock A note on the 2D generalized Zakharov-Kuznetsov equation: Local, global, and scattering results,
\newblock J. Differential Equations 253 (2012), 2558–-2571.


\bibitem{kato}
\newblock T. Kato,
\newblock On the Cauchy problem for the (generalized) Korteweg-de- Vries equations,
\newblock Advances in Mathematics Suplementary Studies, Stud. Appl. Math. 8 (1983), 93--128.

\bibitem{ponce2}
\newblock C. E. Kenig, G. Ponce and L. Vega,
\newblock Well-posedness and scattering results for the generalized Korteweg-de Vries equation
 and the contraction principle,
\newblock Commun. Pure Appl. Math. 46 (1993), 527--620.



\bibitem{larkin}
\newblock N. A. Larkin,
\newblock Korteweg-de Vries and Kuramoto-Sivashinsky Equations in Bounded Domains,
\newblock J. Math. Anal. Appl. 297 (2004), 169--185.

\bibitem{lar2}
\newblock N. A. Larkin, E. Tronco,
\newblock Nonlinear quarter-plane problem for the Korteweg-de Vries equation,
\newblock Electron. J. Differential Equations 2011 (2011), 1--22.

\bibitem{larkintronco}
\newblock N. A. Larkin and E. Tronco,
\newblock Regular solutions of the 2D Zakharov–Kuznetsov equation on a half-strip,
\newblock J. Differential Equations 254 (2013), 81–-101.

\bibitem{pastor}
\newblock F. Linares and A. Pastor,
\newblock Local and global well-posedness for the 2D generalized Zakharov-Kuznetsov equation,
\newblock J. Funct. Anal. 260 (2011), 1060--1085.

\bibitem{pastor2}
\newblock F. Linares, A. Pastor and J.-C. Saut,
\newblock Well-posedness for the ZK equation in a cylinder and on the background of a KdV Soliton,
\newblock Comm. Part. Diff. Equations 35 (2010), 1674--1689.


\bibitem{saut}
\newblock F. Linares and J.-C. Saut,
\newblock The Cauchy problem for the 3D Zakharov-Kuznetsov equation,
\newblock Disc. Cont. Dynamical Systems A 24 (2009), 547--565.

\bibitem{pazy} A. Pazy,
\newblock Semigroups of linear operators and applications to partial
differential equations. \newblock Applied Mathematical Sciences, 44.
Springer-Verlag, New York, 1983.

\bibitem{zuazua}
\newblock G. Perla Menzala, C. F. Vasconcellos and E. Zuazua,
\newblock Stabilization of the Korteweg-de Vries equation with localized damping,
\newblock Quart. Appl. Math. 60 (2002), 111--129.

\bibitem{rosier}
\newblock L. Rosier,
\newblock Exact boundary controllability for the Korteweg-de Vries equation on a bounded domain,
\newblock ESAIM Control Optim. Calc. Var. 2 (1997), 33--55.

\bibitem{rosier1}
\newblock L. Rosier,
\newblock A survey of controllability and stabilization results for partial differential equations,
\newblock RS - JESA 41 (2007), 365--411.

\bibitem{rozan}
\newblock L. Rosier and B.-Y. Zhang,
\newblock Control and stabilization of the KdV equation: recent progress,
\newblock J. Syst. Sci. Complexity 22 (2009), 647--682.


\bibitem{temam}
\newblock J.-C. Saut and R. Temam,
\newblock An initial boundary-value problem for the Zakharov-Kuznetsov equation,
\newblock Advances in Differential Equations 15 (2010), 1001--1031.

\bibitem{temam2}
\newblock J.-C. Saut, R. Temam and C. Wang,
\newblock An initial and boundary-value problem for the Zakharov-Kuznetsov equation in a bounded domain,
\newblock J. Math. Phys. 53 115612 (2012).



\bibitem{zk}
\newblock V. E. Zakharov and E. A. Kuznetsov,
\newblock On three-dimensional solitons,
\newblock Sov. Phys. JETP 39 (1974), 285--286.





\end{thebibliography}
\end{document}